\documentclass[12pt,a4paper]{article}
\usepackage[latin1]{inputenc}
\usepackage[british]{babel}
\usepackage[left=80pt, right=80pt, top=60pt]{geometry}

\usepackage[affil-it]{authblk} 
\usepackage{graphics}
\usepackage{amsfonts}
\usepackage{enumerate}
\usepackage{amssymb}
\usepackage{xfrac}	
\usepackage{slashed} 
\usepackage{savesym}
\savesymbol{intertext}
\usepackage{amsmath}
\restoresymbol{TT}{intertext}
\usepackage{upgreek}
\usepackage[bookmarksopen,bookmarksdepth=2]{hyperref}
\usepackage{threeparttable}

\mathchardef\ordinarycolon\mathcode`\:
\mathcode`\:=\string"8000
\begingroup \catcode`\:=\active
  \gdef:{\mathrel{\mathop\ordinarycolon}}
\endgroup

\newcommand{\N}{\mathbb N}

\newcommand{\R}{\mathbb R}
\newcommand{\C}{\mathbb C}
\newcommand{\D}{\mathcal{D}}

\renewcommand{\S}{\mathcal{S}}
\newcommand{\F}{\mathcal{F}}

\newcommand{\mR}{\mathcal{R}}
\newcommand{\Cr}{C_\mR}
\newcommand{\CR}{C_\mR}
\newcommand{\Sr}{\S_\mR}
\newcommand{\SR}{\S_\mR}
\newcommand{\mB}{\mathcal{B}}

\newcommand{\mC}{\mathfrak{C}}
\newcommand{\Dhr}{\hat{\D}_\mR}

\newcommand{\supnorm}[1]{\norm{ #1 }_\infty}

\newcommand{\norm}[1]{\left\| #1 \right\|}
\newcommand{\nrm}[2]{\left\|#1\right\|_{#2}}

\newcommand{\p}[2]{\langle #1 , #2 \rangle}
\newcommand{\dtzero}[1]{\left.\frac d{dt}#1\right|_{t=0}}

\newcommand{\dd}{\textnormal{d}}
\newcommand{\ran}{\operatorname{ran}}
\newcommand{\nul}{\operatorname{nul}}
\newcommand{\spn}{\textnormal{span}}

\newcommand{\Graff}{\textnormal{Graff}}
\newcommand{\QW}{Q_\hbar^W}
\newcommand{\QB}{Q_\hbar^B}
\newcommand{\QR}{\tilde{Q}_\hbar^W}

\newcommand{\dslash}[1]{\!\slashed{\textnormal{d}}#1\,}

\renewcommand{\mid}{~\middle|~}

\newcommand{\Exp}{\sqrt{\textnormal{Exp}}}	
\newcommand{\wexp}{\Exp}

\newcommand{\MLP}[3]
					{\mathcal{N}_\varphi^{#1}(#2\def\mycmd{#3}\if1\mycmd \else ,#3 \fi)}

\usepackage{amsthm}

\newtheorem{thm}{Theorem}[section]
\newtheorem{lem}[thm]{Lemma}
\newtheorem{prop}[thm]{Proposition}

\newtheorem{defn}[thm]{Definition}

\newtheorem{remark}[thm]{Remark}

\begin{document}

\title{Quantization and the Resolvent Algebra\footnote{This is an author generated post-print of:
T. D. H. van Nuland (2019). Quantization and the Resolvent Algebra. \textit{J. Funct. Anal.} \textbf{277}, issue 8, pages 2815--2838.\\ Incorporated in this post-print is the corrigendum submitted in 2024.}}

\author{Teun D.H. van Nuland}
\affil{E-mail: teunvn@gmail.com}
\date{}

\maketitle

\begin{abstract}
We introduce a novel commutative C*-algebra $C_\mathcal{R}(X)$ of functions on a symplectic vector space $(X,\sigma)$ admitting a complex structure, along with a strict deformation quantization that maps a dense subalgebra of $C_\mathcal{R}(X)$ to the resolvent algebra introduced by Buchholz and Grundling \cite{BG2008}. The associated quantization map is a field-theoretical Weyl quantization compatible with the work of Binz, Honegger and Rieckers \cite{BHR}. We also define a Berezin-type quantization map on all of $C_\mathcal{R}(X)$, which continuously and injectively maps it onto a dense subset of the resolvent algebra.

The commutative resolvent algebra $C_\mathcal{R}(X)$, generally defined on a real inner product space $X$, intimately depends on the finite dimensional subspaces of $X$. We thoroughly analyze the structure of this algebra in the finite dimensional case by giving a characterization of its elements and by computing its Gelfand spectrum.

\end{abstract}

\section{Introduction}
\label{intro}
The resolvent algebra is a C*-algebra modelling the canonical commutation relations. Just as the better known Weyl (C*-)algebra is characterized by the canonical commutation relations in exponentiated form, the resolvent algebra is characterized by the corresponding relations between resolvents. This simple change turns out to give the resolvent algebra a much richer structure, and makes it better suited for modelling dynamics, compared to the Weyl algebra. The resolvent algebra, introduced and thoroughly investigated by Buchholz and Grundling in \cite{BG2008}, appears to be useful for many aspects of quantum mechanics and quantum field theory, but has left us one important question. This question, posed by Buchholz in a personal communication, concerns the classical limit of the resolvent algebra, or, equivalently (at least within a C*-algebraic framework), its emergence from strict deformation quantization theory. We answer this question in this paper, in a way that seamlessly extends quantization in the setting of the compact operators to infinite dimensions.

To explain this, we will view the resolvent algebra in light of different quantization procedures, the first of which is the procedure introduced by Weyl. 
In a C*-algebraic framework \cite{KL98,Rieffel1993,Rieffel1994}, Weyl quantization typically starts from a dense subalgebra of $C_0(\R^{2n})$, like the Schwartz space $\S(\R^{2n})$, which is then mapped to a dense subalgebra of the compact operators $K(L^2(\R^n))$. The Weyl quantization of $f\in\S(\R^{2n})$ is the operator on $L^2(\R^n)$ defined by
\begin{align}\label{QW}
	\QW(f)=\int_{\R^{2n}}\frac{\textnormal{d}^{2n}x}{(2\pi)^n}\hat{f}(x)e^{i\phi(x)}\,,
\end{align}
where $\hbar\in\R$, $\phi(x)$ is a linear combination of position and momentum operators with coefficients $x_1,\ldots x_{2n}$, and $\hat{f}$ is the Fourier transform of $f$. (The precise definitions can be found in Section \ref{sct:Quantization}.) Motivated by quantum field theory, we also wish to quantize functions on an infinite dimensional phase space $X$. Because $\S(X)=0$ for infinite dimensional $X$, our suggestion is to replace $\S(X)$ by 
	$$\Sr(X):=\spn\left\{g\circ P_V\mid g\in\S(V),V\subseteq X \textnormal{ finite dimensional}\right\}\,,$$
where $P_V$ is the (orthogonal) projection onto $V$. Our generalization of Weyl quantization is then defined as
\begin{align}\label{QW nieuw}
	\QW(g\circ P_V):=\int_V \frac{\textnormal{d}^{r}x}{(2\pi)^{r/2}}\hat{g}(x)e^{i\phi(x)}\,,
\end{align}
where $r=\dim V$ and $\phi(x)=i\sqrt{\hbar}(a^*(x)-a(x))$ is a linear combination of creation and annihilation operators (other conventions such as $\phi(x)=\sqrt{\hbar}(a^*(x)+a(x))$ or $\phi(x)=\hbar(a^*(x)+a(x))$ work as well).

Definition \eqref{QW nieuw} relates well to other generalizations of Weyl quantization. Firstly, when $X=\R^{2n}$, \eqref{QW nieuw} is related to \eqref{QW} by a unitary, where we observe that \eqref{QW} is defined for a large class of functions $f$ \cite{Folland,Rieffel1993}, in particular for $f=g\circ P_V$. Secondly, \eqref{QW nieuw} is formally an extension of the quantization map on the Weyl algebra, as defined in \cite{BHR}. Indeed, suppose that $g(P_V(y))=e^{ix\cdot y}$. Then \eqref{QW nieuw} formally dictates $\QW(g\circ P_V)=e^{i\phi(x)}$, and these operators $e^{i\phi(x)}$ generate the Weyl algebra in the Fock representation. In fact, Binz, Honegger and Rieckers proved in \cite{BHR} that this field-theoretical Weyl quantization constitutes a strict deformation quantization, with the almost periodic functions (the C*-algebra generated by the functions $y\mapsto e^{ix\cdot y}$) on the classical side, and the Weyl algebra on the quantum side.

This paper proves the following new fact. Weyl quantization gives a strict deformation quantization of $\Sr$, and the image of $\Sr$ under $\QW$ is a dense subalgebra of the resolvent algebra. In particular, this result entails a continuous field of C*-algebras over $[0,1]$, with the resolvent algebra $\mR(X,\sigma)$ as the constant fiber above $(0,1]$, and $\Cr(X):=\overline{\Sr(X)}$ as the fiber above $0$.

The known continuous field of C*-algebras $\{A_\hbar\}_{\hbar\geq0}$, where $A_0=C_0(\R^{2n})$ and $A_\hbar=K(L^2(\R^n))$ for each $\hbar\in(0,1]$, only applies to finite dimension. As we have now extended this by $A_0\subseteq \Cr(\R^{2n})$, and $A_\hbar\subseteq\mR(\R^{2n},\sigma)$ for $\hbar\in(0,1]$, we can view the resolvent algebra as an elegant way to work in arbitrary dimension.

In addition to Weyl quantization, we also study Berezin quantization in the setting of the resolvent algebra. It turns out that Berezin quantization, defined by extension of
	$$\QB(g\circ P_V):=\int_V \frac{\textnormal{d}^{r}x}{(2\pi)^{r/2}}e^{-\frac{\hbar}{4}\|x\|^2}\hat{g}(x)e^{i\phi(x)}\,,$$
is a continuous positive linear injection $\QB:\Cr(X)\rightarrow\mR(X,\sigma)$ with dense range, which is equivalent to Weyl Quantization in the sense of \cite{KL98}. Again, and perhaps more clearly in this second quantization procedure, $\Cr(X)$ is seen to be the classical limit of $\mR(X,\sigma)$. We will therefore refer to $\Cr(X)$ as the \emph{commutative resolvent algebra}, and devote most of this article to an investigation of its structure.

It turns out that the commutative resolvent algebra is generated by the functions $h^\lambda_x(y):=1/(i\lambda-x\cdot y)$. This gives an equivalent, more direct definition
	$$\Cr(X):=C^*\left(h^\lambda_x\mid\lambda\in\R\setminus\{0\},x\in X\right),$$
	which also relates nicely to the definition of the resolvent algebra. In this way, $h^\lambda_x$ is the classical analogue of the generators $R(\lambda,x):=(i\lambda-\phi(x))^{-1}$ of $\mR(X,\sigma)$. 
	
	This analogy between classical and quantum can be quite useful. Many results of the resolvent algebra turn out to have a classical analogue, which can be understood through simple geometric pictures. For instance, for linearly independent $x,y\in X$, the result (from \cite{BG2008})
\begin{align}\label{distance resolvents}
	\norm{R(1,x)-R(1,y)}\geq1\,,
\end{align}
has a very easy classical counterpart
\begin{align}\label{distance resolvent functions}
	\supnorm{h^1_x-h^1_y}\geq1\,.
\end{align}
	
	An important aspect of the commutative resolvent algebra is that
	$$\Cr(X)=\varinjlim\Cr(V)\,,$$
for finite dimensional $V\subseteq X$, where the connecting maps defining the colimit are $P_V^*:\Cr(V)\rightarrow\Cr(W)$ for $V\subseteq W$. This is one of the reasons why much of our analysis is done on finite-dimensional spaces $X$.

In the last two sections of this paper we show the power and flexibility of the commutative resolvent algebra on $\R^m$. Arbitrary elements of $\Cr(\R^{m})$ are infinite sums of functions $g\circ P_V$, usually converging conditionally. We will make clear how these sums can be decomposed into a finite number of unconditional parts. We will end up with a characterization of the elements of $\Cr(\R^m)$ that behaves nicely with respect to its algebraic structure.

The Gelfand spectrum is a useful way of understanding a commutative C*-algebra. For this reason we will identify the Gelfand spectrum of $\Cr(\R^m)$  with the set of affine subspaces of $\R^m$, equipped with an interesting new topology. We characterize this topology by a convergence criterion as well as by providing a neighborhood basis. Either way, one easily identifies $\R^m$ with its $0$-dimensional affine subspaces. Thusly, we show that the Gelfand spectrum of $\Cr(\R^m)$ is a compactification of $\R^m$.

This paper is organized as follows. Section 2 gives the main definitions and the crucial results. These include a direct definition of $\Cr$ and a proof that $\Sr$ is a Poisson *-algebra. For the reader with a specific interest it is useful to know that Sections \ref{sct:Quantization}, \ref{sct: Function Spaces} and \ref{sct: Gelfand Spectrum} each depend solely on Section \ref{sct: Commutative Resolvent Algebra} and are independent otherwise. In Section \ref{sct:Quantization}, we discuss Weyl quantization, Berezin quantization, and the Resolvent algebra in the setting of Fock space. Section \ref{sct: Function Spaces} characterizes the elements of $\Cr(\R^m)$ in a way that suits its algebraic structure. Finally, Section \ref{sct: Gelfand Spectrum} establishes a precise characterization of the Gelfand spectrum of $\Cr(\R^m)$.



\section{Commutative Resolvent Algebra}
\label{sct: Commutative Resolvent Algebra}
Let $X$ be a real inner product space. We define the commutative resolvent algebra $\Cr(X)$, similar to the definition of the resolvent algebra $\mR(X,\sigma)$ of Buchholz and Grundling \cite{BG2008}, but without assuming the existence of a compatible symplectic structure $\sigma$ on $X$. The classical analogues of the resolvents $R(\lambda,x)$ (defined in \cite{BG2008}, and equivalently in our Section \ref{sct:Quantization},) are the functions
	$$h^\lambda_x(y):=1/(i\lambda-x\cdot y)\qquad (y\in X)\,,$$
	for $x\in X$, $\lambda\in\R\setminus\{0\}$. The inner product $\cdot$ gives rise to a norm $\norm{~}$ and a topology (the standard ones for real pre-Hilbert spaces $X$), making $h^\lambda_x$ a continuous function.

\begin{defn}\label{def: The Commutative Resolvent Algebra}
	The \textbf{commutative resolvent algebra} on $X$, denoted by $\Cr(X)$, or simply by $\Cr$, is the C*-subalgebra of $C_b(X)$ generated by the functions $h^\lambda_x$.
\end{defn}
This C*-algebra $\Cr$ is unital, since $ih^1_{0}=1.$ Let us write $h^\lambda_x=g^\lambda\circ p_x$ for $g^\lambda:=1/(i\lambda-\cdot)$ 
and $$p_x(y):=x\cdot y\,.$$ For $x\neq0$, the pull-back $p_x^*:C_0(\R)\rightarrow C_b(X)$ is an isometric *-homomorphism, allowing for an equivalent definition of $\Cr$. Indeed, the Stone--Weierstrass theorem gives $C^*(g^\lambda|\lambda\in\R\!\setminus\!\{0\}) = C_0(\R)$, implying $C^*(h_x^\lambda|\lambda\in\R\!\setminus\!\{0\})=p_x^*(C_0(\R))$, for any $x$. Hence, $\Cr$ is the C*-algebra generated by $\left\{g\circ p_x\mid g\in C_0(\R), x\in X\right\}$.

 We will see that these $g\circ p_x$ generate more general functions $g\circ P_V$, when we generalize $p_x$ by the (orthogonal) projection $P_V$ onto any finite dimensional subspace $V\subseteq X$, and let $g\in C_0(V)$. It will sometimes be useful to assume that $g$ is a Schwartz function, by which we mean $g\in\S(V)$.

\begin{lem}\label{lem:product of levees}
	Let $V_1,V_2\subseteq X$ be linear and $g_i\in C_0(V_i)$ for $i\in\{1,2\}$. Then
	\begin{enumerate}[(i)]
	    \item $(g_1\circ P_{V_1})(g_2\circ P_{V_2})=g\circ P_{V_1+V_2}$ for some $g\in C_0(V_1+V_2)$,
	    \item if $g_1$ and $g_2$ are both Schwartz, then $g$ is Schwartz as well.
    \end{enumerate}
\end{lem}
\begin{proof}
	Since $P_{V_i}=P_{V_i}\circ P_{V_1+V_2}$, we find $(g_1\circ P_{V_1})(g_2\circ P_{V_2})=g\circ P_{V_1+V_2}$ for $g:=(g_1\circ P_{V_1})(g_2\circ P_{V_2})_{\upharpoonright V_1+V_2}$. Now decompose $V_1+V_2=U_1\oplus U_2\oplus U_3$ for linear $U_i\subseteq X$ with $V_1=U_1\oplus U_3$, $V_2=U_2\oplus U_3$. Then $$g(u_1+u_2+u_3)=g_1(u_1+u_3)g_2(u_2+u_3)$$ for $u_i\in U_i$. When $g_1$ and $g_2$ are Schwartz, bounding the Schwartz norms of $g$ (with respect to any reasonable basis of $U_1\oplus U_2\oplus U_3$) is a matter of applying the general Leibniz rule. This gives $(ii)$, and by density of $\S$ in $C_0$, $(i)$ follows.
\end{proof}

We will relate the structure of $\Cr$ to the functions of the form $g\circ P$, so let us give this type of functions a name.

\begin{defn}
    A \textbf{levee} $f:X\rightarrow\C$ is a composition $f=g\circ P$ of some finite dimensional projection $P$ and some function $g\in C_0(\ran P)$.
\end{defn}

The terminology is explained in Section \ref{sct: Function Spaces}, and illustrated by Figure \ref{fig:1}. Instead of the term levee, one could call a function of the form $g\circ P$ cylindrical, relating to cylindrical sets and measures. However, this might cause confusion, as the term `cylindrical function' in some contexts refers to a Bessel function. 

Levees for which $g$ is Schwartz will be very useful when working with Weyl quantization. We therefore define
    $$\Sr(X):=\spn\left\{g\circ P \text{ levee}\mid g\in \S(\ran P)\right\}\,.$$

This space of finite sums of levees turns out to be an algebra.

\begin{prop}\label{prop: Finite Sr is dense in Cr}
    The space $\Sr(X)$ is a dense *-subalgebra of $\Cr(X)$.
\end{prop}
\begin{proof}
   Let $g\circ P_V$ be a levee with $g\in\S(V)$, and fix an orthonormal basis $v_1,\ldots,v_k$ of $V$. Because the algebraic tensor product $\S(\R)^{\otimes k}$ is densely embedded in $\S(V)$ (with respect to the Schwartz topology), we may assume that $g(t_1v_1+\ldots +t_kv_k)=g_1(t_1)\cdots g_k(t_k)$ for $g_i\in\S(\R)$. We obtain $g\circ P_V=\prod g_i\circ p_{v_i}\in\Cr$ and conclude that $\Sr\subseteq\Cr$.\\
    The set $\Sr$ is clearly closed under linear combinations and involution. Furthermore, closure under multiplication follows by Lemma \ref{lem:product of levees}, and we may conclude that $\Sr$ is a *-subalgebra.\\
    Finally, any generator $h^\lambda_x$ is approximated by functions $g\circ p_x\in\Sr$ where $g\in\S(\R)$ approximates $g^\lambda=1/(i\lambda-\cdot)\in C_0(\R)$. This proves density.
\end{proof}
%
\subsection{$\Cr$-functions at large scale} We will specify the behavior of an arbitrary function $f\in \Cr(X)$ at infinity. To this purpose, we assume $X$ is finite dimensional, but we will shortly see that this gives us information about the general case as well. Let $V+w\subseteq X$ be an affine subspace, with space of directions $S(V):=\left\{v\in V\mid\norm{v}=1\right\}$ when $V\neq\{0\}$, and $S(\{0\}):=\{0\}$. We equip $S(V)$ with the spherical measure $\mu$. The convergence at infinity of $f$ is captured by the following lemma.
\begin{lem}\label{lem: f^V+w constant ae}
	Take $f\in \Cr(X)$ for a finite dimensional $X$. Then the limit 
	\begin{align}\label{f^Vw}
	    f^{V,w}(v):= \lim_{t\rightarrow\infty} f(tv+w)
	\end{align}
	exists for all $v\in S(V)$ and hence defines a function $f^{V,w}:S(V)\rightarrow\C$. Furthermore, $f^{V,w}$ takes a constant value $\mu$-almost everywhere. If $f=g\circ P$ is a levee, then this value is $0$ if $V\nsubseteq \ker P$, and this value is $g(Pw)$ if $V\subseteq \ker P$.
\end{lem}
\begin{proof}
	If $f$ is a levee, then the lemma can be checked by a case distinction.  So when $(f_i)\subseteq \Sr(X)$ converges to $f\in\Cr(X)$, then we have a well-defined function $f_i^{V,w}$ with $f_i^{V,w}=c_i$ $\mu$-a.e. for some $c_i\in\C$. The sequence $(c_i)$ converges to some $c\in\C$, because $(f_i)$  is Cauchy in $\supnorm{\cdot}$. If
		$$\Gamma:=\{v\in S(V)\,|\,\forall{i}\!: f_i^{V,w}(v)=c_i\}\,,$$
	then $\mu(S(V)\setminus\Gamma)=0$ by countable additivity of $\mu$. Now for arbitrary $v\in\Gamma$ we have
	    $$ \lim_{i\rightarrow\infty}\lim_{t\rightarrow\infty} f_i(tv+w)=\lim_{i\rightarrow\infty}c_i=c\,,$$
	and for any $v\in S(V)$ we have $f_i(tv+w)\rightarrow f(tv+w)$ uniformly in $t$. Therefore, $f^{V,w}$ is a function with $f^{V,w}=c$ $\mu$-a.e.
\end{proof}

Apart from its use in Sections \ref{sct: Function Spaces} and \ref{sct: Gelfand Spectrum}, this lemma shows us that levees that are constant in different directions are linearly independent. Indeed, suppose that $\sum_{i=1}^k g_i\circ P_{V_i}=0$ for levees $g_i\circ P_{V_i}$ with distinct $V_i$. One can assume without loss of generality that $X=V_1+\ldots+V_k$ (thereby reducing to the finite dimensional case), and see that $$0=\big(\sum g_i\circ P_{V_i}\big)^{V,w}=g_j(w)\qquad\mu\text{-a.e.,}$$ for $V=V_j^\perp$ a maximal element of $\{V_1^\perp,\ldots,V_k^\perp\}$. It then follows that each $g_i\circ P_{V_i}=0$. 

Thanks to this linear independence, it is allowed to linearly extend a function defined on levees $g\circ P$, as long as this definition is linear in $g$. This will precisely be the case for our two quantization maps in Section \ref{sct:Quantization}.

\subsection{Poisson structure}\label{sct: Poisson Structure}
If $X$ has a compatible complex structure (and therefore in particular a symplectic structure), we can put a canonical Poisson structure on $\Sr(X)$ (which will be necessary for strict deformation quantization). Equipping $X=\R^{2n}$ with the symplectic structure $\sigma(x,y)=\sum x_{n+j}y_j-x_jy_{n+j}$, the *-algebra $\Sr(\R^{2n})$ is a Poisson subalgebra of $C^\infty(\R^{2n})$, because any partial derivative of a levee is again a levee. Let $\{\cdot,\cdot\}_{2n}$ be the Poisson bracket of $C^\infty(\R^{2n})$, and hence of $\Sr(\R^{2n})$. If a compatible hermitian form is fixed on $\R^{2n}$, then any surjective, continuous, partial isometry $p:X\rightarrow\R^{2n}$ (for any $n\in\N$) induces a Poisson structure on the image of $p^*:\Sr(\R^{2n})\rightarrow\Sr(X)$. In effect, we are defining
	$$\{f\circ p,g\circ p\}:=\{f,g\}_{2n}\circ p\,.$$
It can be shown that this bracket does not depend on $p$, using the equivariance of $\{\cdot,\cdot\}_{2n}$ under symplectic transformations and the tensor product. This gives us a Poisson structure on the whole of $\Sr(X)$, because any three levees $g_1\circ P_{V_1}$, $g_2\circ P_{V_2}$, $g_3\circ P_{V_3}$ are in the image of a single $p^*$, namely for the $p$ with $(\ker p)^\perp\supseteq V_1+V_2+V_3$.


\section{Quantization}
\label{sct:Quantization}
Now that we have introduced the classical setup (postponing the more advanced results until Sections 4 and 5), we will connect classical with quantum. We begin by defining the core concepts we need.

\subsection{Fock space, the resolvent algebra, and quantization}
\label{sct:Fock Space and the Resolvent Algebra}
Let $(X,\sigma)$ be a symplectic vector space admitting a compatible hermitian form $\p{\cdot}{\cdot}$ (i.e. $\sigma(x,y)=2\Im\p{x}{y}$, and set $x\cdot y := \Re\p{x}{y}$). For analyzing the resolvent algebra, Buchholz and Grundling use a field $\phi$ that can be defined in multiple ways. Since the field $\phi$ plays an important role for field theoretic quantization as well as for the resolvent algebra, let us define $\phi$ concretely in the setting of Fock space. That is, our Hilbert space is $\F(\overline{X})$, the bosonic Fock space (symmetric Hilbert space) of the completion of $X$ with respect to its complex inner product. We refer to \cite{Guichardet} for details on the Fock space, as well as for the details of the following remarks. Throughout this section, we fix $\hbar\in\R\setminus\{0\}$. As a common domain for $\phi(x)$, we take
	$$\mC:=\spn\left\{\Exp(w):=\sum_{k=0}^\infty\frac{\otimes^k w}{\sqrt{k!}}\mid w\in\overline{X}\right\}\,,$$
on which we define, for any $x\in X$,
	$$\phi(x)\Exp(w):=i\sqrt{\hbar}\left(\p{x}{w}\Exp(w)-\dtzero{\Exp(w+tx)}\right)\,,$$
where the derivative exists because of $(\Exp(w),\Exp(z))=e^{\p{w}{z}}$. Note that our inner product is linear in the second argument, as opposed to \cite{Guichardet}. By Stone's theorem and some calculation, one obtains 
\begin{align}\label{e^(i phi(x)) in Fock rep}
	e^{i\phi(x)}\Exp(w)=e^{-\tfrac{\hbar}{2}\norm{x}^2-\sqrt{\hbar}\p{x}{w}}\Exp(w+\sqrt{\hbar}x)\,.
\end{align}
 It then also follows that $\phi(x)$ is essentially self-adjoint. From \eqref{e^(i phi(x)) in Fock rep} it can easily be checked that
	\[e^{i\phi(x)}e^{i\phi(y)}=e^{-\tfrac{i\hbar}{2}\sigma(x,y)}e^{i\phi(x+y)}\,.\]
The resolvent algebra was defined by Buchholz and Grundling in \cite{BG2008}, through abstract relations. We write this definition down in Fock representation, where it becomes shorter and more suited to our purposes.
\begin{defn}\label{def: The Resolvent Algebra}
	The \textbf{resolvent algebra} $\mR(X,\sigma)$ is the C*-subalgebra of $B(\F(\overline{X}))$ generated by the resolvents $R(\lambda,x):=(i\lambda-\phi(x))^{-1}$ for $\lambda\in\R\!\setminus\!\{0\}$ and $x\in X$.
\end{defn}
The resolvent algebra can be thought of as the commutative resolvent algebra, with the functions $y\mapsto x\cdot y$ replaced by the operators $\phi(x)$. This analogy alone can already yield helpful intuition, as exemplified by \eqref{distance resolvents} and \eqref{distance resolvent functions}. However, to establish a rigorous relation between the two algebras, we will use the concept of strict deformation quantization, defined below. This definition is equivalent to \cite[Definition 1.1.1]{KL98}, other definitions of strict deformation quantization are reviewed in \cite{Hawkins}.

\begin{defn}\label{Def: sdq}
	Let $\tilde{A}_0$ be a complex Poisson algebra, densely contained in a C*-algebra $A_0$, with $\{f,g\}^*=\{f^*,g^*\}$. A \textbf{strict deformation quantization} of $\tilde{A}_0$ consists of a subset $I\subseteq\R$ with $0\in I\cap\overline{I\!\setminus\!\{0\}}$, a collection of C*-algebras $\{A_\hbar\}_{\hbar\in I}$ (with norms $\nrm{\cdot}{\hbar}$), and a collection of injective linear *-preserving maps $$\{Q_\hbar:\tilde{A}_0\rightarrow A_\hbar\}_{\hbar\in I}\,,$$ such that $Q_0$ is the identity map, $Q_\hbar(\tilde{A}_0)$ is a dense *-subalgebra of $A_\hbar$ (for each $\hbar\in I$), and for all $f,g\in\tilde{A}_0$:
	\begin{gather}
        \hbar\mapsto \nrm{Q_\hbar(f)}{\hbar}\text{ is continuous on $I$,}\tag{I}\label{sdq1}\\
        \lim_{\hbar\rightarrow0}\nrm{Q_\hbar(f)Q_\hbar(g)-Q_\hbar(fg)}{\hbar}=0\,,\tag{II}\label{sdq2}\\
        \lim_{\hbar\rightarrow0}\nrm{\tfrac{i}{\hbar}[Q_\hbar(f),Q_\hbar(g)]-Q_\hbar(\{f,g\})}{\hbar}=0\tag{III}\label{sdq3}\,.
    \end{gather}
\end{defn}

\begin{remark}
Under further assumptions formulated in \cite[Theorem 1.2.4]{KL98}, Definition \ref{Def: sdq} implies that $\{A_\hbar\}_{\hbar\in I}$ is a continuous field of C*-algebras. 
We will show that this definition can be applied to the resolvent algebra: in our setting both Weyl and Berezin quantization form a strict deformation quantization inducing a continuous field of C*-algebras (this continuous field of C*-algebras is the same for Weyl and Berezin quantization). This follows from combining Theorem \ref{thm: sdq}, Lemma \ref{lem: QB continuous in hbar>0}, Lemma \ref{Asymptotics Weyl and Berezin}, and Theorem \ref{thm:Berezin} with \cite[Theorem 1.2.4]{KL98}.
\end{remark}

\subsection{Weyl quantization}\label{sct: Weyl Quantization}


We are interested in a generalization of Weyl quantization, defined on a dense subset of the commutative resolvent algebra. As argued in the introduction, it makes intuitive sense to define Weyl quantization of a levee $g\circ P_V\in\Sr(X)$ as
\begin{align}\label{def Weyl quantization}
	\QW(g\circ P_V):=\int_V \dslash{x}\hat{g}(x)e^{i\phi(x)}\,,
\end{align}
and we will now explain how to mathematically interpret this definition. We write ~$\dslash{y}:=(2\pi)^{-m/2}\text{d}^my$ whenever $y$ runs over an $m$-dimensional space, in particular simplifying the notation of our Fourier transform, $\hat{g}(x)=\int_V\dslash{y}g(y)e^{-ix\cdot y}.$ All operator-valued integrals in this section are of the form $\int \text{d}\mu(x)A_x$, where $\mu$ is a finite complex measure and $x\mapsto A_x$ is strongly continuous. These can be defined by $(\int \text{d}\mu(x) A_x)\psi:=\int \text{d}\mu(x)A_x\psi$, where the latter integral is a (complex) Pettis integral. In our situation $\text{d}\mu(x)=\dslash{x}\hat{g}(x)$, and $A_x = e^{i\phi(x)}$, so our expression \eqref{def Weyl quantization} is defined. Notice that the $\hbar$-dependence of $\QW$ comes from $\phi$.

We can linearly extend \eqref{def Weyl quantization}, as was discussed after the proof of Lemma \ref{lem: f^V+w constant ae}. Furthermore, any $\QW(f)$ is bounded because of the estimation
\begin{align}\label{QW bounded by L1-norm}
	\norm{\QW(g\circ P_V)}\leq\nrm{\hat{g}}{1}\,,
\end{align}
and we therefore have a well-defined linear map $\QW:\Sr(X)\rightarrow B(H)$. In order to show that this map determines a strict deformation quantization, our main task is to prove that $\QW(\Sr(X))$ is a dense *-subalgebra of $\mR(X,\sigma)$.

To this purpose, we first restrict ourselves to $V=\spn\{x\}$. In this case, it is nicer to replace $P_V$ by the function $p_x:y\mapsto x\cdot y$. If $g\in\S(\R)$, then the levee $g\circ p_x\in\Sr(X)$ is quantized by the operator
	$$\QW(g\circ p_x)=\int_\R \dslash{t} \hat{g}(t)e^{it\phi(x)}\,,$$
    which turns out to behave nicely as a function of $g$.
    
    \begin{prop}\label{lem:Phi continuous *-hom}
    	Weyl quantization applied to levees of one variable coincides with the continuous functional calculus of $\phi(x)$. That is,
    	$$\QW(g\circ p_x)=g(\phi(x))\,.$$
    \end{prop}
    \begin{proof}
    	Define $\rho:L^1(\R)\rightarrow B(\F(\overline{X}))$ by $\rho(f):=\int\dslash{t}f(t)e^{it\phi(x)}$. Let $C^*(\R)$ be the group C*-algebra of $\R$, with associated norm $\nrm{\cdot}{*}$. The definition of $\nrm{\cdot}{*}$ in particular implies $\norm{\rho(f)}\leq\nrm{f}{*}$, giving us a continuous extension $\rho: C^*(\R)\rightarrow B(\F(\overline{X}))$. It is known that the Fourier transform $\hat{\cdot}:C_0(\R)\rightarrow C^*(\R)$ is continuous, and therefore $\rho(\hat{\cdot}):C_0(\R)\rightarrow B(\F(\overline{X}))$ is continuous as well. Since $\rho(\hat{g})=\QW(g\circ p_x)$ ($g\in\S(\R)$), we are left to show that $\rho(\hat{\cdot})$ coincides with the functional calculus of $\phi(x)$. With the basic rules for Fourier transforms, $\rho(\hat{\cdot})$ can be shown to be a *-homomorphism on a dense subset, and hence on all of $C_0(\R)$.  Furthermore, from \cite[Corollary 4.4]{BG2008} it straightforwardly follows that $\rho((1/(i\lambda - \cdot))\hat{~})=(i\lambda-\phi(x))^{-1}=R(\lambda,x)$. This completes the proof.
    \end{proof}

\begin{prop}\label{prop: QW *-subalgebra}
	The set $\QW(\Sr(X))$ is a *-subalgebra of $B(H)$.
\end{prop}
\begin{proof}
	Let $V_1,V_2\subseteq X$ be finite dimensional subspaces, and $g_i\in\S(V_i)$. Decompose $V_1+V_2=U_1\oplus U_2\oplus U_3$ for linear $U_i\subseteq X$ with $V_1=U_1\oplus U_3$, $V_2=U_2\oplus U_3$. 
	We find that
	\begin{align*}
		\QW(g_1\circ P_{V_1})\QW(g_2\circ P_{V_2})&=\int_{V_1}\dslash{x}\!\!\int_{V_2}\dslash{y}\hat{g_1}(x)\hat{g_2}(y)e^{\frac{-i\hbar}{2}\sigma(x,y)}e^{i\phi(x+y)}\\
		&=\int_{U_1}\!\!\dslash{u_1}\!\!\!\int_{U_2}\!\!\dslash{u_2}\!\!\!\int_{U_3}\!\!\dslash{u_3}\hat{g}(u_1+u_2+u_3)e^{i\phi(u_1+u_2+u_3)},
	\end{align*}
	where we have defined, for $u_i\in U_i$,
	\begin{align*}
		\hat{g}(u_1+u_2+u_3):=\int_{U_3}\dslash{u_3'}\hat{g_1}(u_1+u_3-u_3')\hat{g_2}(u_1+u_3')e^{\frac{-i\hbar}{2}\sigma(u_1+u_3-u_3',u_2+u_3')}\,.
	\end{align*}
	By bounding the Schwartz norms of $\hat{g}$ with respect to an appropriate basis, one finds that $\hat{g}\in\S(V_1+V_2)$ and therefore $g\circ P_{V_1+V_2}\in\Sr(X)$. Hence $$\QW(g_1\circ P_{V_1})\QW(g_2\circ P_{V_2})=\QW(g\circ P_{V_1+V_2})\in\QW(\Sr(X))\,.$$  One easily sees that $\QW(\overline{g\circ P_V})=\QW(g\circ P_V)^*$, so the proposition follows.
\end{proof}

\begin{thm}\label{thm:QW(Sr) dense in mR}
	We have $\overline{\QW(\Sr(X))}=\mR(X,\sigma)$, for every $\hbar\neq0$.
\end{thm}
\begin{proof}
	Since $\QW(\Sr)$ is a *-algebra, we want its closure to contain $R(\lambda,x)$. Take a sequence $(g_j)$ in $\S(\R)$ converging uniformly to $1/(i\lambda-\cdot)$. Then $g_j(\phi(x))$ converges to $R(\lambda,x)$, and therefore $R(\lambda,x)\in\overline{\QW(\Sr)}$ by Proposition \ref{lem:Phi continuous *-hom}. By Proposition \ref{prop: QW *-subalgebra} we may conclude that $\mR(X,\sigma)\subseteq\overline{\QW(\Sr)}$.

	 We are left to show that $\QW(g\circ P_V)\in\mR(X,\sigma)$ for every levee $g\circ P_V\in\Sr(X)$. We do this by induction in $\dim V$. We choose a unit vector $v\in V$ and write elements of $V$ as $tv+w$, for $t\in\R$ and $w\perp v$. Let $f:V\rightarrow S^1$ be the function such that $e^{i\phi(tv+w)}=f(tv+w)e^{it\phi(v)}e^{i\phi(w)}$. Notice that the span of functions of the form
		 $$tv+w\mapsto g_1(t)g_2(w)\qquad (t\in\R,w\perp v)$$
	 lies dense in $\S(V)$ with respect to the Schwartz topology. Because of \eqref{QW bounded by L1-norm}, it suffices to assume that $\hat{g}f$ is of this form, and we therefore write $(\hat{g}f)(tv+w)=\hat{g_1}(t)\hat{g_2}(w)$ for some $g_1\in\S(\R)$, $g_2\in\S(\{v\}^\perp)$. By virtue of Proposition \ref{lem:Phi continuous *-hom} we find that
	 \begin{align*}
	 	\QW(g\circ P_V)&=\int_{\R}\dslash{t}\int_{\{v\}^\perp}\dslash{w}\hat{g_1}(t)\hat{g_2}(w)e^{it\phi(v)}e^{i\phi(w)}\\
		&=g_1(\phi(v))\int_{\{v\}^\perp}\dslash{w}\hat{g_2}(w)e^{i\phi(w)}\,.
	 \end{align*}
	By the induction hypothesis the latter integral is in $\mR(X,\sigma)$, and by the Stone--Weierstrass theorem we can approximate $g_1$ by polynomials in $1/(i\lambda-\cdot)$. Functional calculus now gives $g_1(\phi(v))\in\mR(X,\sigma)$, thus proving that $\QW(g\circ P_V)\in\mR(X,\sigma)$.
\end{proof}

We will now relate our generalization of Weyl quantization $\QW$ to the usual finite dimensional Weyl quantization, and show why this gives us a strict deformation quantization.

By usual convention, Weyl quantization of a suitable function $f:\R^{2n}\rightarrow\C$ is
\begin{align}\label{QR}
\QR(f):=\int\dslash{x}\hat{f}(x)e^{i\sum(x_jP_j+x_{n+j}Q_j)}\,,
\end{align}
where $P_j\psi(y)=-i\hbar\frac{\partial\psi}{\partial y_j}$ and $Q_j\psi(y)=y_j\psi(y)$. 

Rieffel, in his memoir \cite{Rieffel1993}, defines a very broad generalization of Weyl quantization, and, in Chapter 9, discusses strict deformation quantization. In particular, as is written in \cite[Chapter 1]{Rieffel1994}, \eqref{QR} is defined and determines a strict deformation quantization of $\mB(\R^{2n})$, where by definition $f\in\mB(\R^{2n})$ is a smooth function all of whose derivatives of all degrees are bounded.
In that same chapter, Rieffel writes $\overline{\mB_\hbar}$ for the completion of $\mB$ with respect to some new C*-norm $\nrm{\cdot}{\hbar}$ and equiped with a different product $\times_\hbar$. Around \cite[equation (1.3)]{Rieffel1994}, Rieffel shows that a certain map $\overline{\mB_\hbar}\rightarrow B(L^2(\R^n))$, $f\mapsto L_f$ is a *-representation, and that the inclusion $\mB\hookrightarrow\overline{\mB_\hbar}$ is a strict deformation quantization. What we call $\QR$ is actually the composition of this inclusion and $f\mapsto L_f$.

We will now relate the strict deformation quantization map $\QR$ to our $\QW$. Let us fix a continuous surjective partial isometry $p:X\rightarrow\R^{2n}$, as we have done in Section \ref{sct: Poisson Structure}. Now $(\ker p)^\perp\rightarrow B(\F(\overline{X}))$, $x\mapsto e^{i\phi(x)}$ is a representation of the canonical commutation relations in exponential form, so by the Stone--von Neumann theorem there exists a subspace $W\subseteq \F(\overline{X})$ invariant under $\{e^{i\phi(x)}\}_{x\perp\ker p}$, together with a unitary $U:W\rightarrow L^2(\R^n)$ such that
\begin{align}\label{Stone von Neumann}
	e^{i\phi(x)}=U^*e^{i\sum p(x)_jP_j+p(x)_{n+j}Q_j}U\,.
\end{align}
Now for every $f\in\Sr(\R^{2n})\subseteq\mB(\R^{2n})$, we have
\begin{align}
	\QW(f\circ p)&=\int_{(\ker p)^\perp}\dslash{x}\hat{f}(px)e^{i\phi(x)}\nonumber\\
	&=U^*\int_{\R^{2n}}\dslash{x}\hat{f}(x)e^{i\sum x_jP_j+x_{n+j}Q_j}U\equiv U^*\QR(f)U\,.\label{QW vs QR}
\end{align}

This link between $\QW$ and $\QR$ can now be used to prove this paper's most crucial result.

\begin{thm}\label{thm: sdq}
	Let ${A_0:=\Cr(X)}$ and ${A_\hbar:=\mR(X,\sigma)}$ for ${\hbar\neq0}$. Then ${I=\R}$, together with the collection of C*-algebras $\{A_\hbar\}_{\hbar\in I}$, and the maps $\QW:\Sr(X)\rightarrow\mR(X,\sigma)$, constitute a strict deformation quantization of $\Sr(X)$.
\end{thm}
\begin{proof}
	We already know that $\QW$ is linear and *-preserving. For checking injectivity and \eqref{sdq1}, \eqref{sdq2} and \eqref{sdq3} of Definition \ref{Def: sdq}, we choose $f,g\in\Sr(X)$. Then we can find a surjective continuous partial isometry $p$ with $f,g\in p^*\Sr(\R^{2n})$, and apply \eqref{QW vs QR}.
	 
	 The last remaining requirement of Definition \ref{Def: sdq} is that $\QW(\Sr)$ is a dense *-subalgebra of $A_\hbar$. This is exactly the statement that we have worked towards. For $\hbar=0$ it follows from Proposition \ref{prop: Finite Sr is dense in Cr}, and for $\hbar\neq0$ it is a combination of Proposition \ref{prop: QW *-subalgebra} and Theorem \ref{thm:QW(Sr) dense in mR}.
\end{proof}

This result rigorously establishes the commutative resolvent algebra as the classical limit of the resolvent algebra. It is a welcome fact that quantization in the setting of the resolvent algebra can be done with Weyl quantization, about which much is known, also in the infinite dimensional case \cite{BHR,Weaver}.

We will now prove a similar result for Berezin quantization, which allows us to quantize the entire commutative resolvent algebra.

\subsection{Berezin quantization}\label{sct: Berezin Quantization}

Let $\hbar>0$. Define, for an arbitrary levee $g\circ P_V\in\Sr(X)$, its Berezin quantization $\QB(g\circ P_V)$ by
\begin{align}\label{QB}
	\QB(g\circ P_V):=\int_V\dslash{x}e^{-\frac{\hbar}{4}\norm{x}^2}\hat{g}(x)e^{i\phi(x)}\,.
\end{align}
The discussion after Lemma \ref{lem: f^V+w constant ae} justifies us in linearly extending this map. 

Using the partial isometry $p:X\rightarrow\R^{2n}$ and the unitary $U:W\rightarrow L^2(\R^{n})$ from before (satisfying \eqref{Stone von Neumann}), it can be shown that the operator $U\QB(f\circ p)U^*$ coincides with the Berezin quantization of $f\in \Sr(\R^{2n})\subseteq L^\infty(\R^{2n})$ as defined by Landsman in \cite[Section II.2.3]{KL98}. As a consequence of \cite[Theorem II.1.3.5]{KL98}, $\QB$ uniquely extends to a continuous positive linear map $$\QB:\Cr(X)\rightarrow B(\F(\overline{X}))\,.$$

Let us introduce the space $\hat{\D}(V)$ of Schwartz functions $f\in\S(V)$ of which the Fourier transform $\hat{f}$ is compactly supported. We also need the associated space
	$$\Dhr(X):=\spn\left\{g\circ P\textnormal{ levee}\mid g\in\hat{\D}(\ran P) \right\}\,.$$
We define the suggestively written operator $e^{\frac{\hbar}{4}\Updelta}:\Dhr(X)\rightarrow\Dhr(X)$ by linear extension of
\begin{align}\label{exp laplace}
	e^{\frac{\hbar}{4}\Updelta}(g\circ P):= (e^{-\frac{\hbar}{4}\norm{\cdot}^2}\hat{g})\check{~}\circ P\,,
\end{align}
where $\check{~}$ denotes the inverse Fourier transform. The notation $e^{\frac{\hbar}{4}\Updelta}$ is justified when $X=\R^m$ and $\Updelta=\sum_j\partial_j^2$ is the Laplace operator. It should be clear that $e^{\frac{\hbar}{4}\Updelta}$ is a bijection. Furthermore, \eqref{exp laplace} immediately gives us
\begin{align}\label{link Berezin and Weyl}
	\QB(f)=\QW(e^{\frac{\hbar}{4}\Updelta}f)\,,
\end{align}
for all $f\in\Dhr(X)$, and therefore $\QB(\Dhr(X))=\QW(\Dhr(X))$.
In fact, because $\hat{\D}$ lies dense in $\S$ with respect to the Schwartz topology, equation \eqref{QW bounded by L1-norm} implies that
	$$\overline{\QB(\Dhr(X))}=\overline{\QW(\Sr(X))}\,,$$
and we conclude, by Theorem \ref{thm:QW(Sr) dense in mR}, that
	$$\overline{\QB(\Dhr(X))}=\mR(X,\sigma)\,.$$

We now proceed to prove \eqref{sdq1} and \eqref{sdq2} for Berezin quantization.
\begin{lem}\label{lem: QB continuous in hbar>0}
	For any $f\in\Cr(X)$, the function $\hbar\mapsto \QB(f)$ is continuous on $(0,\infty)$.
\end{lem}
\begin{proof}
	For a levee $g\circ P\in\Dhr(X)$ we have, as we have seen,
	\begin{align*}
		\QB(g\circ P)&=\QW((e^{-\frac{\hbar}{4}\norm{\cdot}^2}\hat{g})\check{~}\circ P)\,.
	\end{align*}
	Furthermore, $\QW(f)=Q^W_1(f_\hbar)$, when we define $f_\hbar(x):=f(\sqrt{\hbar}x)$. We obtain $$\QB(g\circ P)=Q^W_1((e^{-\frac{\hbar}{4}\norm{\cdot}^2}\hat{g})\check{~}\circ P)_\hbar)=Q^W_1((e^{-\frac{\norm{\cdot}^2}{4}}\hat{g_\hbar})\check{~}\circ P)\,.$$ A bit of analysis yields that $\hbar\rightarrow g_\hbar$ is Schwartz-continuous on $(0,\infty)$. Therefore, by \eqref{QW bounded by L1-norm}, we find that $\hbar\rightarrow\QB(g\circ P)$ is continous. By continuity of $\QB$ and density of $\Dhr(X)\subseteq\Cr(X)$, the lemma follows.
\end{proof}

It should be stressed that the continuity of $\hbar\mapsto \QB(f)$ rests on our specific definition of $\phi(x)$. For example, let $X=\R^2$, and introduce the quantum mechanical operators $Q\psi(x):=x\psi(x)$ and $P^\hbar\psi(x)=-i\hbar\frac{d}{dx}\psi(x)$ on $L^2(\R)$, the above result remains valid when setting $\phi(x):=\sqrt{\hbar}(x_1P^1+x_2Q)$ and defining $\QB$ through \eqref{QB}. If we use $\phi(x):=x_1P^\hbar+x_2Q$ in \eqref{QB}, we recover Landsman's definition of $\QB$ in \cite[Section II.2.3]{KL98}. These two possible definitions of $\phi(x)$ are related by an $\hbar$-dependent unitary transformation. One should be warned that, when using the latter formula for $\phi(x)$, Lemma \ref{lem: QB continuous in hbar>0} is no longer true. A similar comment applies to Weyl quantization.

\begin{lem}\label{Asymptotics Weyl and Berezin}
	Weyl and Berezin quantization are equivalent in the sense that, for all $f\in\Sr(X)$, the map $$\hbar\mapsto\norm{\QW(f)-\QB(f)}$$ is continuous on $I=[0,\infty)$.
\end{lem}
\begin{proof}
	For a levee $g\circ P\in\Sr(X)$ we have
		$$\QB(g\circ P)-\QW(g\circ P)=\QW((e^{-\frac{\hbar}{4}\norm{\cdot}^2}\hat{g}-\hat{g})\check{~}\circ P)\,.$$
	Because $\hbar\mapsto e^{-\frac{\hbar}{4}\norm{\cdot}^2}\hat{g}$ is $L^1$-continuous, the bound \eqref{QW bounded by L1-norm} implies that we have $$\norm{\QB(g\circ P)-\QW(g\circ P)}\rightarrow0$$ as $\hbar\rightarrow0$. For $\hbar>0$, we can apply an argument similar to the proof of Lemma \ref{lem: QB continuous in hbar>0}.
\end{proof}

The final result of this section gives the C*-algebraic relation between the commutative resolvent algebra and the resolvent algebra.
\begin{thm}\label{thm:Berezin}
Let $(X,\sigma)$ be a symplectic vector space compatible with a hermitian structure. Let $\tilde A_0:=\Dhr(X)$, $A_0=\CR(X)$ and $A_\hbar:=\mR(X,\sigma)$ for $\hbar\neq0$. Then $I=\R$, together with the collection of C*-algebras $\{A_\hbar\}_{\hbar\in I}$, and the maps $$Q^B_\hbar:\Dhr(X)\to\mR(X,\sigma),$$
constitute a strict deformation quantization of $\Dhr(X)$. 
Moreover, $Q^B_\hbar$ uniquely extends to a continuous linear map $$Q^B_\hbar:\CR(X)\to\mR(X,\sigma),$$ which is positive and injective. The latter map is not surjective. In fact, for finite dimensional $X$, not all compact operators are in $\QB(C_0(X))$.
\end{thm}
\begin{proof}
	We have already seen that $\QB$ is positive and continuous, and that $\overline{\QB(\Dhr(X))}=\mR(X,\sigma)$. The proof that $\QB(\Dhr(X))$ is a *-algebra is analogous to the proof of Proposition \ref{prop: QW *-subalgebra}. Concerning \eqref{sdq1}, we note that continuity of $\hbar\mapsto\QB(f)$ for $\hbar>0$ is achieved by Lemma \ref{lem: QB continuous in hbar>0}, and because $\lim_{\hbar\rightarrow0}\norm{\QW(f)}=\supnorm{f}$, the same holds for $\QB$ by Lemma \ref{Asymptotics Weyl and Berezin}. Similarly, \eqref{sdq2} holds for $\QW$, so by Lemma \ref{Asymptotics Weyl and Berezin} also for $\QB$. To obtain \eqref{sdq3} for $f,g\in\Dhr(X)$, one writes $f=f_n\circ p$, $g=g_n\circ p$ for a surjective continuous partial isometry $p:X\rightarrow\R^{2n}$ as we have done in \textsection\ref{sct: Poisson Structure}. One then calculates the Fourier transforms of $e^{\frac{\hbar}{4}\Updelta}\{f,g\}=e^{\frac{\hbar}{4}\Updelta}\{f_n,g_n\}_{2n}\circ p$ and $\{e^{\frac{\hbar}{4}\Updelta}f,e^{\frac{\hbar}{4}\Updelta}g\}=\{e^{\frac{\hbar}{4}\Updelta}f_n,e^{\frac{\hbar}{4}\Updelta}g_n\}_{2n}\circ p$, to see that the respective functions on $\R^{2n}$ get arbitrarily close in $L^1$-norm as $\hbar\rightarrow0$.
	
We are left to prove injectivity and the two final statements concerning non-surjectivity. The following discussion is quite technical and was missed in an earlier version of this paper.\footnote{See also the corrigendum to the published version [T. D. H. van Nuland (2019). Quantization and the Resolvent Algebra. \textit{J. Funct. Anal.} \textbf{277}, issue 8, pages 2815--2838].} The author thanks Lorenzo Pettinari for pointing this out and substantially contributing to the remainder of this proof.
	
	\paragraph{Injectivity.} Injectivity of $\QB:\Dhr(X)\to\mR(X,\sigma)$ is easily obtained: Because $\QW:\Sr(X)\to\mR(X,\sigma)$ is injective, and $e^{\frac{\hbar}{4}\Updelta}:\Dhr(X)\to\Sr(X)$ is injective, its composition $\QB:\Dhr(X)\to\mR(X,\sigma)$ is injective as well (note that we could also have defined $\QB$ on $\Sr(X)$ with the same result). This concludes the proof of strict deformation quantization. However, injectivity of the continuous extension $\QB:\CR(X)\to\mR(X,\sigma)$ is more difficult.
	
	We use the Fock space vectors $\wexp(w)$ $(w\in X)$ of Section \ref{sct:Quantization}, which satisfy
$$e^{i\phi(x)}\wexp(w)=e^{-\frac{\hbar}{2}\|x\|^2-\sqrt{\hbar}\p{x}{w}}\wexp(w+\sqrt{\hbar}x),$$
and
\begin{align}\label{eq:inner product exponential vectors}(\wexp(v),\wexp(w))=e^{\p{v}{w}}.
\end{align}
A short computation shows that, for every levee $g\circ P_V\in\hat{\mathcal D}_\mR(X)$,
\begin{align}\label{eq:matrix element}
(\wexp(-w),\QB(g\circ P_V)\wexp(w))=e^{-\|w\|^2}\int_V\dslash x \hat{g}(x)\exp\left(-\frac{3\hbar}{4}\|x\|^2-2\sqrt{\hbar}(w\cdot x)\right),
\end{align}
where we are using that the pre-Hilbert space $X$ and the real inner product space $X$ are related by $x\cdot y :=\Re(\p{x}{y})$.

We use the Fourier theorem $\int\dslash x\hat{g}(x)F(x)=\int\dslash y g(y)\hat{F}(y)$. The Fourier transform of $F\in\S(V)$ where $F(x)=\exp\left(-\frac{3\hbar}{4}\|x\|^2-2\sqrt{\hbar}(w\cdot x)\right)$ is obtained by a computation that is standard apart from the appearance of $P_V$:
\begin{align*}
\hat{F}(y)&=e^{\frac{4}{3}\|P_V w\|^2}\int_V \dslash x e^{-i(x\cdot y)}\exp(-\tfrac{3\hbar}{4}\|x\|^2-2\sqrt{\hbar}(x\cdot w)-\tfrac{4}{3}\|P_V w\|^2)\\
&=e^{\frac{4}{3}\|P_V w\|^2}\int_V \dslash x e^{-i(x\cdot y)}\exp(-\tfrac{3\hbar}{4}\|x+\tfrac{4}{3\sqrt{\hbar}}P_V w\|^2)\\
&=e^{\frac{4}{3}\|P_V w\|^2}e^{i(\frac{4}{3\sqrt{\hbar}}P_Vw\cdot y)}\int_V \dslash x e^{-i(x\cdot y)}\exp(-\tfrac{3\hbar}{4}\|x\|^2)\\
&=e^{\frac{4}{3}\|P_V w\|^2}e^{i(\frac{4}{3\sqrt{\hbar}}P_Vw\cdot y)}\left(\tfrac{2}{3\hbar}\right)^{\frac{\dim V}{2}}e^{-\frac{1}{3\hbar}\|y\|^2}.
\end{align*}
We obtain the formula
\begin{align}\label{eq:formula V}
&(\wexp(-w),\QB(g\circ P_V)\wexp(w))\nonumber\\
&=e^{-\|w\|^2}\int_V\frac{\dd y}{\sqrt{3\pi\hbar}^{\dim V}} g(y)\exp\left(-\tfrac{1}{3\hbar}\|y\|^2+\tfrac{4}{3\sqrt{\hbar}}i(w\cdot y)+\tfrac{4}{3}\|P_Vw\|^2\right).
\end{align}
For any finite dimensional subspace $W\subseteq X$ containing $V$ we can write $W=V\oplus U$ for $V\perp U$. Inserting the identity
\begin{align}\label{eq:identity}
1=\int_U \frac{\dd y}{\sqrt{3\pi\hbar}^{\dim U}}\exp(-\tfrac{1}{3\hbar}\|y\|^2+\tfrac{4}{3\sqrt{\hbar}}i(w\cdot y)+\tfrac{4}{3}\|P_U w\|^2),
\end{align}
we may generalize \eqref{eq:formula V} to any finite dimensional subspace $W$ containing $V$, namely,
\begin{align}\label{eq:matrix element for every W}
&(\wexp(-w),\QB(g\circ P_V)\wexp(w))\nonumber\\
&=e^{-\|w\|^2}\int_{W}\frac{\dd y}{\sqrt{3\pi\hbar}^{\dim W}} (g\circ P_V)(y)\exp(-\tfrac{1}{3\hbar}\|y\|^2+\tfrac{4}{3\sqrt{\hbar}}i(w\cdot y)+\tfrac{4}{3}\|P_{W}w\|^2).
\end{align}

With the purpose of obtaining a contradiction, we fix $f\in\CR(X)\setminus\{0\}$ such that $\QB(f)=0$. Let $x_0\in X$ be such that $|f(x_0)|>0$. We shall show that there exists a cylinder set on which the values of $f$ stay close to $f(x_0)$ (and are in particular bounded away from zero).
Fix an $f_0\in\SR(X)$ such that $\|f-f_0\|_\infty<\frac{1}{4}|f(x_0)|$, and fix a finite-dimensional subspace $V_0\subseteq X$ such that  $f_0\circ P_{V_0}=f_0$ and $x_0\in V_0$. As $f_0|_{V_0}$ is continuous, there exists a neighborhood $K\subseteq V_0$ of $x_0$ on which $|f_0(v)-f_0(x_0)|<\frac{1}{4}|f(x_0)|$ for all $v\in K$. Hence, $|f_0(x)-f_0(x_0)|=|f_0(P_{V_0}(x))-f_0(x_0)|<\frac{1}{4}|f(x_0)|$ for all $x\in P_{V_0}^{-1}(K)$. It follows that
\begin{align*}
|f(x)-f(x_0)|&\leq |f(x)-f_0(x)|+|f_0(x)-f_0(x_0)|+|f_0(x_0)-f(x_0)|\\
&<\frac{3}{4}|f(x_0)|,
\end{align*}
for all $x$ in the cylinder set $P_{V_0}^{-1}(K)$.

We now fix a sequence $(f_n)_{n\in\N}\subseteq \hat{\mathcal D}_\mR(X)$ converging uniformly to $f$, and pick finite-dimensional subspaces $W_n$ such that $V_0\subseteq W_n$ and $f_n\circ P_{W_n}=f_n$. From \eqref{eq:matrix element for every W} and linearity we obtain
\begin{align*}
&(\wexp(-w),\QB(f_n)\wexp(w))\\
&=e^{-\|w\|^2}\int_{W_n}\frac{\dd y}{\sqrt{3\pi\hbar}^{\dim W_n}} f_n(y)\exp(-\tfrac{1}{3\hbar}\|y\|^2+\tfrac{4}{3\sqrt{\hbar}}i(w\cdot y)+\tfrac{4}{3}\|P_{W_n}w\|^2).
\end{align*}
It follows from \eqref{eq:inner product exponential vectors} that $w\mapsto\wexp(w)$ is continuous, so the above expression is continuous in $w$, as well as bounded uniformly in $w$. We now integrate both sides over $w\in V_0$ against the Schwartz function $w\mapsto \hat{h}(w)e^{-\frac{1}{3}\|w\|^2}$, where $h\in\S(V_0)$ is arbitrary. By Tonelli--Fubini, we obtain
\begin{align*}
&\int_{V_0}\dslash w (\wexp(-w),\QB(f_n)\wexp(w))\hat{h}(w)e^{-\frac{1}{3}\|w\|^2}\\
&=\int_{W_n}\frac{\dd y}{\sqrt{3\pi\hbar}^{\dim W_n}} f_n(y)e^{-\frac{1}{3\hbar}\|y\|^2}\int_{V_0}\dslash w \,e^{\frac{4}{3\sqrt{\hbar}}i(w\cdot y)}\hat{h}(w)\\
&=\int_{W_n}\frac{\dd y}{\sqrt{3\pi\hbar}^{\dim W_n}} f_n(y)e^{-\frac{1}{3\hbar}\|y\|^2}h(\tfrac{4}{3\sqrt{\hbar}}P_{V_0}(y)),
\end{align*}
for all $h\in\S(V)$. The absolute value of the left-hand side can be bounded uniformly in $n$.
Since $\QB$ is supremum norm to operator norm continuous, and $\QB(f)=0$, it follows that the left-hand side converges to 0. Therefore, the right-hand side converges to 0.
Furthermore, by using a special case of \eqref{eq:identity}, namely,
\begin{align}\label{eq:Gaussian integration identity}
\int_{W_n}\frac{\dd y}{\sqrt{3\pi\hbar}^{\dim W_n}} e^{-\frac{1}{3\hbar}\|y\|^2}=1,
\end{align}
the triangle inequality for integrals implies that
\begin{align*}
&\Bigg|\int_{W_n}\frac{\dd y}{\sqrt{3\pi\hbar}^{\dim W_n}} f_n(y)e^{-\frac{1}{3\hbar}\|y\|^2}h(\tfrac{4}{3\sqrt{\hbar}}P_{V_0}(y))\\
&\qquad-\int_{W_n}\frac{\dd y}{\sqrt{3\pi\hbar}^{\dim W_n}} f(y)e^{-\frac{1}{3\hbar}\|y\|^2}h(\tfrac{4}{3\sqrt{\hbar}}P_{V_0}(y))\Bigg|\\
&\leq\|f_n-f\|_\infty\cdot\|h\|_\infty\to 0.
\end{align*}
It follows that
\begin{align}\label{eq:limit is 0}
\lim_{n\to\infty}\int_{W_n}\frac{\dd y}{\sqrt{3\pi\hbar}^{\dim W_n}} f(y)e^{-\frac{1}{3\hbar}\|y\|^2}h(\tfrac{4}{3\sqrt{\hbar}}P_{V_0}(y))=0.
\end{align}
By construction of $V_0$ there is a cylinder set $P_{V_0}^{-1}(K)$ on which $|f(x)-f(x_0)|<\tfrac34|f(x_0)|.$ Pick a positive Schwartz function $h\in\S(V_0)$ satisfying $\operatorname{supp} h\subseteq \frac{4}{3\sqrt{\hbar}}K$ and
\begin{align*}
\int_{W}\frac{\dd y}{\sqrt{3\pi\hbar}^{\dim W}}e^{-\frac{1}{3\hbar}\|y\|^2}h(\tfrac{4}{3\sqrt{\hbar}}P_{V_0}(y))=1,
\end{align*}
for $W=V_0$ and hence (by \eqref{eq:identity}) for $W=W_n$ for all $n\in\N$.
We obtain that
\begin{align*}
&\left|\left(\int_{W_n}\frac{\dd y}{\sqrt{3\pi\hbar}^{\dim W_n}}f(y)e^{-\frac{1}{3\hbar}\|y\|^2}h(\tfrac{4}{3\sqrt{\hbar}}P_{V_0}(y))\right)-f(x_0)\right|\\
&=\left|\int_{W_n}\frac{\dd y}{\sqrt{3\pi\hbar}^{\dim W_n}}\left(f(y)-f(x_0)\right)e^{-\frac{1}{3\hbar}\|y\|^2}h(\tfrac{4}{3\sqrt{\hbar}}P_{V_0}(y))\right|\\
&\leq\int_{P_{V_0}^{-1}(K)}\frac{\dd y}{\sqrt{3\pi\hbar}^{\dim W_n}}\left|f(y)-f(x_0)\right|e^{-\frac{1}{3\hbar}\|y\|^2}h(\tfrac{4}{3\sqrt{\hbar}}P_{V_0}(y))\\
&\leq\frac{3}{4}|f(x_0)|,
\end{align*}
for every $n\in\N$. So the integral within brackets on the left-hand side is uniformly bounded away from zero by $\frac{1}{4}|f(x_0)|$ and hence cannot converge to 0. Comparing with \eqref{eq:limit is 0} yields a contradiction. Hence $f=0$ for every $f\in\CR(X)$ with $\QB(f)=0$.
\paragraph{Proof that $\QB:\CR(X)\to\mR(X,\sigma)$ is not surjective.}
We let $g\circ P_V\in \hat{\mathcal D}_\mR(X)$ be a levee, and write $g_\lambda(x):=g(\lambda x)$ for $x\in X$, $\lambda>0$. For the Fourier transform we have $\widehat{g_\lambda}(x)=\lambda^{-\dim V}\hat{g}(\frac{x}{\lambda})$, and hence
\begin{align*}
Q^B_\hbar(g_\lambda\circ P_V)&=\int_V\dslash{x} e^{-\frac{\hbar}{4}\|x\|^2}\widehat{g_\lambda}(x)e^{i\phi(x)}\\
&=\lambda^{-\dim V}\int_V\dslash x e^{-\frac{\hbar}{4}\|x\|^2}\hat{g}(\frac{x}{\lambda})e^{i\phi(x)}.
\end{align*}
As $\lambda\to\infty$, the norm
$\|g_\lambda\circ P_V\|_\infty=\|g\|_\infty$ stays constant, whereas the operator norm of $Q_\hbar^B(g_\lambda\circ P_V)$ is bounded by
\begin{align*}
\|Q^B_\hbar(g_\lambda\circ P_V)\|\leq\lambda^{-\dim V}\int_V\dslash{x} e^{-\frac{\hbar}{4}\|x\|^2}\|\hat{g}\|_\infty,
\end{align*}
which goes to 0 as $\lambda\to\infty$. Hence $(Q_\hbar^B)^{-1}$, if it would exist, would not be a bounded operator.
If $Q^B_\hbar:\CR(X)\to\mR(X,\sigma)$ would be bijective, then it would be a bijective bounded linear map between Banach spaces, and hence have a bounded inverse.  We conclude that $Q^B_\hbar:\CR(X)\to\mR(X,\sigma)$ is not surjective.

The same argument applies to $\QB:C_0(\R^{2n})\to K(L^2(\R^n))$, given by taking $X=\R^{2n}$ with standard symplectic and real inner product structure. Since this map is the restriction of $\QB:\CR(\R^{2n})\to\mR(\R^{2n},\sigma)$ to $C_0(\R^{2n})\subseteq\CR(\R^{2n})$, it is also injective, and the same argument as before (taking $g\circ P_V=g\in \hat{\mathcal{D}}(\R^{2n})$ for $V=\R^{2n}$) shows that it is not a Banach isomorphism, so there is a compact operator $C\in K(L^2(\R^n))\setminus\QB(C_0(\R^{2n}))$.
\end{proof}

\section{Function Spaces}\label{sct: Function Spaces}

This section aims to give a concrete description of $\Cr(\R^m)$. The results of Sections \ref{sct: Function Spaces} and \ref{sct: Gelfand Spectrum} are simplified by the finite dimensionality of $\R^m$, while staying applicable in an infinite dimensional setting, since $\Cr(X)$ is the direct limit of $\Cr(\R^m)$, $m\rightarrow\infty$. Another advantage of the finite dimensional case is that it can be visualised in 3D, see Figure \ref{fig:1}.

\begin{figure}[h!]
\centering
\resizebox{0.5\textwidth}{!}{%
  \includegraphics{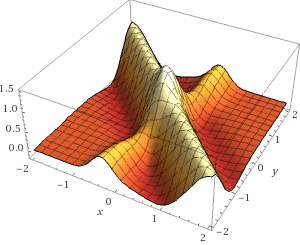}
}
\caption[The sum of two levees.]{The sum of two levees.\footnote{Plotted using Wolfram Alpha LLC.: Wolfram$|$Alpha. \url{http://www.wolframalpha.com/input/?i=plot+0.5e\%5E(-10(x-y\%2F6)\%5E2)\%2Bcos(2x\%2B2y)e\%5E(-(x\%2By)\%5E2),+x\%3D-2+to+2+and+y\%3D-2+to+2} (access July 6, 2018).}}
\label{fig:1}
\end{figure}

\footnotetext[1]{Plotted using Wolfram Alpha LLC: Wolfram$|$Alpha, \url{http://www.wolframalpha.com/input/?i=plot+0.5e\%5E(-10(x-y\%2F6)\%5E2)\%2Bcos(2x\%2B2y)e\%5E(-(x\%2By)\%5E2),+x\%3D-2+to+2+and+y\%3D-2+to+2} (access July 6, 2018).}

For $m=2$ and $\nul P(=\dim \ker P)=1$, the surface plot of the absolute value of $g\circ P$ resembles a physical levee with top height of $\supnorm{g}$ stretching out indefinitely in the direction of $\ker P$ and - in the perpendicular direction - descending into the flat surrounding landscape. The function $g$ determines the shape of the levee and $P$ determines the direction into which it extends. For general values of $\nul P$ and $m$, it is helpful to imagine an affine space of dimension $\nul P$, around which the support of $g\circ P$ is concentrated.
	 
	Proposition \ref{prop: Finite Sr is dense in Cr} displays a dense subset of $\Cr(\R^m)$, consisting solely of finite sums of levees. In fact, one can show that the elements of $\Cr(\R^m)$ are precisely the uniformly convergent series
\begin{align}\label{sum of levees}
	f=\sum_{i=1}^\infty g_i\circ P_i\,.
\end{align}
	The convergence of this sum is conditional, and this makes the representation \eqref{sum of levees} less useful regarding products and sums in $\Cr(X)$. In fact, as this chapter will make clear, if the terms in \eqref{sum of levees} are rearranged, the sum often diverges pointwise. To obtain a useful representation of $f\in\Cr$, avoiding conditionally convergent sums, we will define function spaces $C_r(\R^m)$, consisting of countable sums of levees $g_i\circ P_i$ for which $\nul P_i= r$, modulo levees $g\circ P$ with $\nul P < r$.
	
\begin{defn}\label{defn: C_r}
	For $0\!\leq\! r\!\leq m$, define the spaces $C_r(\R^m)$ as follows. First, $C_0(\R^m)$ is the usual space of continuous functions vanishing at infinity (showing the consistency of our notation). Assuming $C_{r-1}(\R^m)$ is a vector space, we denote the equivalence class of $f\in C_b(\R^m)$ in $C_b(\R^m)/C_{r-1}(\R^m)$ by $[f]_{r-1}$, and use the topology induced by
	    $$\nrm{[f]_{r-1}}{r-1}:=\inf_{\xi\in C_{r-1}}\supnorm{f-\xi}\,.$$
	We define
	$$C_r(\R^m) := \left\{  f\in C_b(\R^m) \!\mid\!\!\! 
	\begin{array}{l}
		[f]_{r-1} = \sum_{i}  [g_i\circ P_i]_{r-1} \text{ for } P_i\text{ distinct (m-r)-} \\
		\text{dimensional projections, and }g_i\in C_0(\ran P_i)
	\end{array}
	\!\right\}\!, $$
	where the sum is over an arbitrary countable set (and hence unconditional).

%
	
\end{defn}

We often write $\nrm{f}{r-1}:=\nrm{[f]_{r-1}}{r-1}$ for convenience. The function spaces $C_r$ build up the commutative resolvent algebra, as made precise by the following theorem. We will postpone its proof for a short while.

\begin{thm} \label{thm: Com Res Alg}
	We have
	$$\Cr(\R^m)=C_{m}(\R^m).$$
	Moreover, $C_0\subset C_1 \subset \ldots \subset C_{m}$ is a chain of closed ideals in $\Cr$.
\end{thm}

%
%

If we want to check whether a given function $f$ is in $C_r(\R^m)$ (and hence in the commutative resolvent algebra), Definition \ref{defn: C_r} demands the existence of a certain unconditionally convergent sum $\sum [g_i\circ P_i]_{r-1}$. It would be convenient if the assumption of unconditional convergence is not necessary, and this is indeed what the following lemma proves.
\begin{lem} \label{prop: norm sum sup norm}
	Let $I\subseteq\N$ be any subset, and let $\sum_{i\in I}[g_i\circ P_i]_r$ be a (possibly conditionally) convergent sum of levees with different $P_i$ of nullity $\nul P_i=r+1$. Then
	\begin{align} \label{eqn: norm sum sup norm}
		\nrm{\sum_{i\in I} [g_i\circ P_i]_r}{r}=\sup_{i\in I} \supnorm{g_i}\,.
	\end{align}
	Hence any such series is unconditionally convergent.
\end{lem}
\begin{proof}
	By continuity of $\nrm{\cdot}{r}$ on $C_b/C_r$, we only need to show \eqref{eqn: norm sum sup norm} for every finite $I\subset \N$. We will use induction on  $\# I$. Let $j\in I$ be such that $\sup_{i\in I}\supnorm{g_i}=\supnorm{g_j}$. Then by the induction hypothesis, $$\hspace{-7pt}\nrm{\sum_{j\neq i\in I} g_i\circ P_i}{r}\hspace{0.5pt}\leq \supnorm{g_j}\,.$$ Fix $\epsilon >0$ and take $\xi\in C_r$ such that 
	\begin{align}\label{eqn: defining property xi}
		\supnorm{\sum_{i\neq j}g_i\circ P_i-\xi}\!\leq\supnorm{g_j}+\epsilon\,.
	\end{align}
	So both $\sum_{i\neq j}g_i\circ P_i-\xi$ and $g_j\circ P_j$ are (almost) bounded by $\supnorm{g_j}$, but their sum may be substantially larger on some region. It turns out that this region is small enough to be corrected for by a $C_r$-function. More precisely, we can find $\phi\in C_r(\R^m)$ such that $$\supnorm{\sum_{i\in I}g_i\circ P_i-\xi-\phi}\leq\supnorm{g_j}+\epsilon\,.$$ Some analysis shows that $$\phi=\left(\sum_{i\neq j}g_i\circ P_i-\xi\right)\frac{|g_j\circ P_j|}{\supnorm{g_j}}$$ does the job. The fact that $\phi\in C_r$ follows from Lemma \ref{lem:product of levees}, using $P_i\neq P_j$ and closedness of $C_r$. We conclude that $\nrm{\sum_{i\in I} g_i\circ P_i}{r}\leq \supnorm{g_j}\,$.
		
	To attain $\supnorm{g_j}$, we choose $w\in\ran P_j$ with $|g_j(w)|=\supnorm{g_j}$, and set $V:=\ker P_j$. With the help of Lemma \ref{lem: f^V+w constant ae}, we find
	\begin{align}\label{apply lemma 2}
		\left(\sum_{i\in I} g_i\circ P_i-\xi\right)^{V,w}=g_j(w)\qquad\mu\text{-a.e.,}
	\end{align}
	because $V\subseteq \ker P_i$ iff $i=j$, and $V\nsubseteq \ker P$ for all levees $g\circ P\in C_r(\R^m)$. From \eqref{apply lemma 2} we obtain the equality $\nrm{\sum g_i\circ P_i}{r}=\supnorm{g_j}$.
	Thus we have finished our inductive step, and the lemma follows. 
	\end{proof}

We can now prove the theorem that relates the spaces $C_r(\R^m)$ to the commutative resolvent algebra.

\begin{proof}[Proof of Theorem \ref{thm: Com Res Alg}]
	Using induction on $r\leq m$, we will prove the following  claim:
	\begin{align} \label{eqn: IH}
		C_r(\R^m)\text{ is a C*-subalgebra of }\Cr(\R^m).
	\end{align}
	If $r=0$ this follows by applying the Stone--Weierstrass theorem, or by recalling that $\Sr\subseteq \Cr$. Suppose now that \eqref{eqn: IH} is true for a fixed $r<m$. Then $C_b/C_r$ is a C*-algebra, and it swiftly follows that $C_{r+1}$ is a *-algebra. The main problem is showing that $C_{r+1}$ is a closed subset of $C_b$.
	
	Let $(f^j)_{j\in\N}\subset C_{r+1}$ converge uniformly to $f$. Write $[f^j]_r=\sum_i[g^j_i\circ P^j_i]_r$ with $g^j_i$ and $P^j_i$ as in Definition \ref{defn: C_r}. Simply taking the limit of each term $g^j_i\circ P^j_i$ is obstructed by the $j$-dependence of $P^j_i$, but this obstruction can be circumvented. We can reshuffle the terms and add zeroes to obtain $\tilde{g}_i^j$ and $P_i$ such that (for all $j\in\N$)
	$$\sum_{i\in\N}[g^j_i\circ P^j_i]=\sum_{i\in \N}[\tilde{g}^j_i\circ P_i]\,.$$
	Lemma \ref{prop: norm sum sup norm} shows us that $(f^j)$ is Cauchy if and only if $(\tilde{g}^j_i)$ is uniformly Cauchy:
	    $$\sup_{i\in \N}\supnorm{\tilde{g}^j_i- \tilde{g}^k_i} = \nrm{\sum_{i\in \N}[(\tilde{g}^j_i- \tilde{g}^k_i)\circ P_i]}{r}=\nrm{f^j-f^k}{r}\rightarrow 0\,.$$
	Thus we may define $g_i:=\lim \tilde{g}^j_i \in C_0(\ran P_i)$. It follows that $\tilde{g}^j_i\rightarrow g_i$ uniformly in $i$.

	Again using Lemma \ref{prop: norm sum sup norm}, convergence of the series $\sum [\tilde{g}^j_i\circ P_i]$ implies $\big\|\tilde{g}^j_i\big\|_{\infty}\rightarrow0$ (for all $j$). Therefore $\supnorm{g_i}\rightarrow0$, which in turn implies convergence of $\sum[g_i\circ P_i]$. Now closedness of $C_{r+1}$ follows from the following calculation. Using Lemma \ref{prop: norm sum sup norm} once more, we have
	\begin{align*}
		\nrm{[f]-\sum_i[g_i\circ P_i]}{r}
		&= \lim_j\nrm{\sum_i[(\tilde{g}^j_i-g_i)\circ P_i]}{r}\\
		&= \lim_j\,\,\sup_i\supnorm{\tilde{g}^j_i-g_i}=0\,.
	\end{align*}
	
	Let $f\in C_{r+1}$ be arbitrary, written as $$[f]=\sum[g_i\circ P_i]\in C_{r+1}/C_r\,,$$ with the usual conventions. Then all $g_i\circ P_i\in \Cr$, and thereby also the partial sums $f^k:=\sum_{i=1}^{k}g_i\circ P_i\in \Cr$. Since $\nrm{f^k-f}{r}\rightarrow0$, we can find $\xi_k\in C_r\subseteq \Cr$ such that $\supnorm{f^k-\xi_k-f}\rightarrow0$. Hence, $f\in \Cr$.
	
	Thus we have proved that $C_{r+1}(\R^m)$ is a C*-subalgebra of $\Cr$. By induction it follows that this holds for all $r<m$, from which the second statement of Theorem \ref{thm: Com Res Alg} follows. We also find $C_{m}(\R^m)\subseteq\Cr(\R^m)$.
	
	The converse inclusion follows if $h^\lambda_x\in C_{m}(\R^m)$ for all $\lambda\neq0,~x\in\R^m$. Define $P$ as the projection onto the span of $x$. Then $\ker P$ is $m$-dimensional when $x=0$ and is $(m-1)$-dimensional otherwise. Since $g(Py):=h^\lambda_x(y)$ defines a function $g\in C_0(\ran P)$, we finally obtain $h^\lambda_x=g\circ P\in C_{m}(\R^m)$.
\end{proof}

We have now obtained a concrete description of $\Cr(\R^m)$ in terms of sums of functions $g\circ P$. Another possible description restricts to one-dimensional projections $P$, but allows to compose with another $C_0$-function. Namely, it turns out that the span of functions of the form $g\circ\sum g_j\circ p_{x_j}$ (with $g\in C_0(\R)$ and a finite sum of real-valued levees $g_j\circ p_{x_j}$) is dense in $\Cr(\R^m)$. An elaborate proof will be given elsewhere, along with envisioned applications for machine learning.


\section{Gelfand Spectrum}\label{sct: Gelfand Spectrum}
We implicitly encountered characters of the commutative resolvent algebra in Lemma \ref{lem: f^V+w constant ae}. Let us now define them precisely. For $V\subseteq\R^m$ linear, $w\in V^\perp$ and $f\in\Cr(\R^m)$, we have defined $f^{V,w}:S(V)\rightarrow\C$ in \eqref{f^Vw}. Let $\chi(V+w)(f)$ be the unique $z\in\C$ such that $f^{V,w}=z$ almost everywhere.\footnote{The character $\chi(V+w)$ can be thought of as the `mean value' on $V+w$.} A quick calculation shows that $\chi(V+w)$ is multiplicative and nonzero, hence $\chi(V+w)\in\Delta(\Cr(\R^m))$, where $\Delta(\Cr(\R^m))$ is the Gelfand spectrum of the commutative resolvent algebra, more briefly denoted by $\Delta$, carrying the weak*-topology (i.e. the Gelfand topology). In practice the characters $\chi(V+w)$ are calculated on levees, where they become
\begin{align*}
	\chi(V+w)(g\circ P)= 
	\begin{cases}
		g(Pw), & \text{if }V\subseteq\ker P,\\
		0, & \text{otherwise.}
	\end{cases}
\end{align*}


What does it mean if a net $(\chi(V_\alpha+w_\alpha))_\alpha$ weak*-converges to $\chi(V+w)$? In that case we have
	$$\chi(V_\alpha+w_\alpha)(g\circ P_{V^\perp})\rightarrow\chi(V+w)(g\circ P_{V^\perp})=g(w)\,,$$
	for any $g\in C_0(V^\perp)$. It follows that eventually (for all $\alpha$ bigger than a fixed $\alpha_0$) we have $V_\alpha\subseteq V=\ker P_{V^\perp}$. Also, by choosing a sequence of $g$'s with support closing in upon $w$, it follows that $P_{V^\perp}w_\alpha\rightarrow w$. Inspired by these results, we will prove (in Theorem \ref{thm: Gelfand Spectrum}) that $\Delta$ is homeomorphic to the following space.
\begin{defn}\label{defn: Omega} 
	We define the set $$\Omega:=\left\{V+w\mid V\subseteq\R^m\text{ linear, }w\in V^\perp\right\}\,,$$ and say that a net $(V_\alpha+w_\alpha)_\alpha$ in $\Omega$ \textbf{is absorbed in} $V+w\in\Omega$ if  $P_{V^\perp}w_\alpha\rightarrow w$ and eventually $V_\alpha\subseteq V$.
\end{defn}
As a set, $\Omega$ is the affine Grassmanian $\Graff(\R^m)$, but we will endow $\Omega$ with a different topology. By the previous discussion, if $\chi(V_\alpha+w_\alpha)\rightarrow\chi(V+w)$, then $V_\alpha+w_\alpha$ is absorbed in $V+w$. Since the converse is false, (as all nets in $\Omega$ are absorbed in $\R^m+0$,) we will define a notion of convergence that is slightly stronger than the notion of absorption.
\begin{defn}\label{defn: Omega topology} 
	A net $(V_\alpha+w_\alpha)_\alpha$ in $\Omega$ \textbf{converges to} $V+w\in\Omega$ iff it is absorbed in $V+w$ and none of its subnets is absorbed in any $\tilde{V}+\tilde{w}\subsetneq V+w$.
\end{defn}
To see that our notion of convergence induces a topology, we prove the following lemma. It also gives useful insight into the structure of the topology of $\Omega$. 

\begin{lem}\label{lem: Omega convergence and topology}
	The notion of convergence in Definition \ref{defn: Omega topology} induces a topology, in which the following statements hold. For every $V+w\in\Omega$, the set $$B_r(V+w):=\left\{V'+w'\subseteq V+w''\mid\norm{w''-w}<r\right\}$$ is open, of which the closure equals
	\begin{align}\label{Br bar}
		\overline{B_r}(V+w):=\left\{V'+w'\subseteq V+w''\mid\norm{w''-w}\leq r\right\}\,.
	\end{align}
	The set $\left\{B_r(V+w)\setminus \bigcup_{i=1}^k\overline{B_{r_i}}(V_i+w_i)\mid k\in\N,~ r,r_i>0,~V_i+w_i\subsetneq V+w\right\}$ is a neighborhood basis of $V+w$. 
\end{lem}
\begin{proof}
	If $\tau_\Omega$ is the set of subsets $U\subseteq\Omega$ such that every converging net outside of $U$ has a limit outside of $U$, then $\tau_\Omega$ is easily seen to be a topology. 
We will now show the last three claims of the lemma hold with respect to $\tau_\Omega$. It subsequently follows that convergence with respect to $\tau_\Omega$ is the same as convergence in the sense of Definition \ref{defn: Omega topology}.
	
	If $V_\alpha+w_\alpha\rightarrow V'+w'\subseteq V+w''$ with $\norm{w''-w}<r$, then we will eventually have $V_\alpha+w_\alpha\subseteq V+P_{V^\perp}w_\alpha$ with $\norm{P_{V^\perp} w_\alpha-w}<r$, showing that $B_r(V+w)$ is open (that is, an element of $\tau_\Omega$).
	
	To show that the set $\overline{B_r}(V+w)$ as defined in \eqref{Br bar} is closed, we choose a net $(V_\alpha+w_\alpha)\subseteq\overline{B_r}(V+w)$ converging to some $V'+w'$. Take the unique $w_\alpha''\in V^\perp$ such that $V_\alpha+w_\alpha\subseteq V+w_\alpha''$ and $\norm{w_\alpha''-w}\leq r$. We will try to find a subnet of $(V_\alpha+w_\alpha)$ that is absorbed in an affine space lying in $V'+w'$. Defining $\tilde{V}:=V'\cap V,$ we already find that eventually $V_\alpha\subset \tilde{V}$. It can be proved (first for $\norm{P_{\tilde{V}^\perp}w_\alpha}=1$, then in general,) that there exists a constant $C$ such that
	\begin{align}\label{eqn: bound with C no0}
		\norm{P_{\tilde{V}^\perp}w_\alpha}\leq C\max(\norm{P_{V'^\perp}w_\alpha},\norm{P_{V^\perp}w_\alpha})\,.
	\end{align}
	To estimate the right-hand-side of \eqref{eqn: bound with C no0}, first observe that $P_{V^\perp}w_\alpha=w_\alpha''$, which is bounded by $r+\norm{w}$. Secondly, observe that $P_{V'^\perp} w_\alpha\rightarrow w'$, so that $(P_{V'^\perp} w_\alpha)$ is eventually bounded. Now \eqref{eqn: bound with C no0} implies that $(P_{\tilde{V}^\perp}w_\alpha)$ has a bounded subnet, and therefore a convergent subnet, denoted by $(P_{\tilde{V}^\perp} w_\beta)\subseteq(P_{\tilde{V}^\perp}w_\alpha)$. This net converges to some $\tilde{w}\in V'+w'$. Hence $(V_\beta+w_\beta)$ is absorbed in $\tilde{V}+\tilde{w}$, and so we must have $\tilde{V}=V'$ and hence $V'\subseteq V$. Because $\norm{P_{V^\perp}w'-w}\leq r$, we find that $V'+w'\in \overline{B_r}(V+w)$. Therefore $\overline{B_r}(V+w)$ is closed, and is the closure of $B_r(V+w)$.
	
	Suppose that $V+w\in U\in\tau_\Omega$. Define $A:=\{V'+w'\subsetneq V+w\}$ and the partially ordered set $I:=\left\{\alpha\subseteq A\mid \#\alpha<\infty\right\}$ with inclusion. We can now prove that there is an $\alpha\in I$ such that 
	\begin{align}\label{eqn: set in U in tau_Omega}
		B_{\frac{1}{1+\#\alpha}}(V+w)\setminus\bigcup_{V'+w'\in\alpha}\overline{B_1}(V'+w')\subseteq U\,,
	\end{align}
	which implies our final claim. Indeed, if there was no such $\alpha$, then we would canonically find a net $(V_\alpha+w_\alpha)_{\alpha\in I}$ outside of $U$ such that every $V_\alpha+w_\alpha$ is in the left-hand-side of \eqref{eqn: set in U in tau_Omega}. It would then easily follow that $V_\alpha+w_\alpha\rightarrow V+w$, giving a contradiction.
\end{proof}

We have a topological embedding $\R^m\hookrightarrow\Omega$ by sending $w\mapsto \{0\}+w$, as a result of Definition \ref{defn: Omega}. This turns out to determine a compactification.

\begin{thm}
	The space $\Omega$ is a compactification of $\R^m$.
\end{thm}
\begin{proof}
	Compactness follows from Definition \ref{defn: Omega topology}. Indeed, to any net $(V_\alpha+w_\alpha)$ we can assign a $V+w\in\Omega$ such that some subnet $(V_\beta+w_\beta)\subseteq(V_\alpha+w_\alpha)$ is absorbed in $V+w$. Either $V_\beta+w_\beta\rightarrow V+w$ or a subsubnet $(V_\gamma+w_\gamma)\subseteq(V_\beta+w_\beta)$ is absorbed in a smaller dimensional affine space. The thus resulting chain of subnets has to stop somewhere, because $\dim V<\infty$, and gives us a convergent subnet of $(V_\alpha+w_\alpha)$.
	
	To show that $\R^m$ is dense in $\Omega$, let $V+w$ be arbitrary, and suppose that every $V'+w'$ with $\dim V'<\dim V$ lies in $\overline{\R^m}$, i.e. the closure of $\R^m$ in $\Omega$. Then we can construct a sequence in $\overline{\R^m}$, converging to $V+w$, as follows. We choose $U\subset V$ with $\dim U=\dim V-1$, some $u\in V\cap U^\perp$, and a sequence $(t_i)\subset\R$ without convergent subsequence. Then $U+t_iu\rightarrow V+w$. Applying induction to the dimension of $V$, it follows that $\overline{\R^m}=\Omega$.
\end{proof}

The topology on $\Omega$ indeed matches the (weak*-)topology on $\Delta$:

\begin{lem}\label{prop:chi embedding}
	The function $\chi:\Omega\rightarrow\Delta$ is an embedding (i.e. a continuous open injection).
\end{lem}
\begin{proof}
	We begin with injectivity. Let $\chi(V+w)=\chi(V'+w')$ for some $V+w,~V'+w'\in\Omega$. Take a projection $P$ onto $V^\perp$ and take a $g\in C_0(V^\perp)$ with $g(w)=1$, and $g(v)<1$ for all $v\neq w$. Now
		$$\chi(V'+w')(g\circ P)=\chi(V+w)(g\circ P)=1\,,$$
	so $V'\subseteq V$ and $g(Pw')=1$. By symmetry we obtain $V'=V$, and therefore $g(w')=1$. It follows that $V+w=V'+w'$.
	
	 We are left to check that the maps $\chi$ and $\chi^{-1}\!:\!\chi(\Omega)\rightarrow\Omega$ preserve convergence of nets.
	 
	Suppose $\chi(V_\alpha+w_\alpha)\rightarrow\chi(V+w)$. As already discussed, $V_\alpha+w_\alpha$ is absorbed in $V+w$. Let $(V_\beta+w_\beta)$ be a subnet that is absorbed in $\tilde{V}+\tilde{w}\subsetneq V+w$. Take a levee $f=g\circ P_{\tilde{V}^\perp}$, where $g(\tilde{w})=1$, so
	$$\lim_\beta\chi(V_\beta+w_\beta)(f)=\lim_\beta g(P_{\tilde{V}^\perp}w_\beta)=1\neq0=\chi(V+w)(f).$$
	This contradicts $\chi(V_\alpha+w_\alpha)\rightarrow\chi(V+w)$. We conclude that $V_\alpha+w_\alpha\rightarrow V+w$.
	
	Suppose conversely that $V_\alpha+w_\alpha\rightarrow V+w$. It is sufficient to prove that
		$$\chi(V_\alpha+w_\alpha)(g\circ P)\rightarrow\chi(V+w)(g\circ P)\,,$$
	for an arbitrary levee $g\circ P\in \Cr$. If $V\subseteq \ker P$, then this follows from a simple computation.
	If $V\nsubseteq\ker P$, then it remains to show that $\chi(V_\alpha+w_\alpha)(g\circ P)$ converges to zero. In the notation of Lemma \ref{lem: Omega convergence and topology}, we eventually have 
	\begin{align}\label{B_R}
		V_\alpha+w_\alpha\in B_1(V+w)\setminus \overline{B_R}(\ker P)
	\end{align}
	for arbitrarily large $R$. Since \eqref{B_R} implies either $V_\alpha\nsubseteq\ker P$ or $\norm{Pw_\alpha}>R$, we find that $\chi(V_\alpha+w_\alpha)(g\circ P)\rightarrow0$, so we are done.
\end{proof}

\begin{thm}\label{thm: Gelfand Spectrum}
	The Gelfand spectrum of the commutative resolvent algebra $\Cr(\R^m)$ is homeomorphic to $\Omega$, i.e. $\Delta(\Cr(\R^m))\cong\Omega$ via the map $\chi$.
\end{thm}
\begin{proof}
	This relies on Lemma \ref{prop:chi embedding}. Continuity of $\chi$ implies that its pullback,
		$$\chi^*:C(\Delta)\rightarrow C(\Omega),\qquad f\mapsto f\circ\chi\,,$$
	is a *-homomorphism. \iffalse{We are left to show injectivity and surjectivity of $\chi^*$. Suppose $\chi^*(\hat{f})=0$ for some $\hat{f}\in C(\Delta)$, which is the Gelfand representation of $f\in \Cr(\R^m)$. For all $w\in\R^m$ we have
	\begin{align*}
		0=\chi^*(\hat{f})(0+w)=\chi(0+w)(f)=f(w).
	\end{align*}
	Hence $\chi^*$ is injective.}\else{As injectivity can be straightforwardly checked, we are left to show surjectivity of $\chi^*$. }\fi If $g\in C(\Omega)$, then $ g\circ\chi^{-1}\in C(\chi(\Omega))$ by Lemma \ref{prop:chi embedding}. Since $\chi(\Omega)$ is a compact subset of the compact Hausdorff space $\Delta$, we may use Urysohn's lemma to extend $g\circ\chi^{-1}$ to $\Delta$. We obtain a function $h\in C(\Delta)$ such that $h\circ\chi=g$, completing the proof.
\end{proof}

\noindent \textbf{Acknowledgements}\\
~\\
I would like to thank Detlev Buchholz for formulating the question that initiated this paper, and Klaas Landsman for his great supervision along the way. For proofreading and/or crucial help I am also much indebted to Karl-Henning Rehren, Marc Rieffel, Arnoud van Rooij, Ruben Stienstra, Walter van Suijlekom, and Marjolein Troost. 

I am greatly indebted to Lorenzo Pettinari for finding two mistakes in arXiv version 1 of this article, which have been fixed with his help in this arXiv version, version 2.\footnote{See also the corrigendum to the published version [T. D. H. van Nuland (2019). Quantization and the Resolvent Algebra. \textit{J. Funct. Anal.} \textbf{277}, issue 8, pages 2815--2838].} He should be credited for the better part of the proof of Theorem \ref{thm:Berezin}.

This work was supported in part by NWO Physics Projectruimte (680-91-101).


\end{document}